%% file: warschaupaper_svjour.tex
\documentclass[smallextended]{svjour3}
\usepackage{graphicx}
\usepackage{mathptmx}
\usepackage{amsfonts,amssymb,amsmath, latexsym}

\usepackage{color}

\usepackage{graphicx}
\usepackage{psfrag}
\usepackage{multirow}

\smartqed

\newcommand{\R}{\mathbb{R}}

\newcommand{\mv}{\,\vert\, }

\newcommand{\longsetto}[1]{\mathop{\longrightarrow}\limits^#1}
\newcommand{\skalp}[1]{\langle #1\rangle}
\newcommand{\xb}{\bar x}

\newcommand{\pb}{\bar p}
\newcommand{\lb}{\bar \lambda}

\newcommand{\xba}{{\bar x^\ast}}
\newcommand{\oo}{o}

\newcommand{\Gr}{{\rm gph\,}}
\newcommand{\dom}{{\rm dom\,}}

\newcommand{\range}{{\rm rge\,}}

\newcommand{\inn}{{\rm int\,}}
\newcommand{\bd}{{\rm bd\,}}

\newcommand{\Limsup}{\mathop{{\rm Lim}\,{\rm sup}}}

\begin{document}
\title{On the Aubin property of a class of parameterized  variational systems}
\author{H. Gfrerer \and J.V. Outrata}
\institute{H. Gfrerer \at Institute of Computational Mathematics, Johannes Kepler University Linz,
              A-4040 Linz, Austria, \email{helmut.gfrerer@jku.at}
              \and J.V. Outrata \at Institute of Information Theory and Automation, Academy of Sciences
              of the Czech Republic, 18208 Prague, Czech Republic, and Centre for
              Informatics and Applied Optimization, Federation University of Australia, POB 663,
              Ballarat,  Vic 3350, Australia,  \email{outrata@utia.cas.cz}}

\maketitle
\begin{abstract}
The paper deals with a new sharp criterion ensuring the Aubin property of solution maps to a class of
parameterized variational systems. This class includes parameter-dependent variational inequalities
with  non-polyhedral constraint sets and also parameterized generalized equations with conic
constraints. The new criterion requires computation of directional limiting coderivatives of the normal-cone mapping for the
so-called critical directions.  The respective formulas have the form of a second-order chain rule and
extend the available calculus of directional limiting objects. The suggested procedure is illustrated
by means of examples.
\end{abstract}

\section{Introduction}

In \cite{GO3}, the authors have developed a new sufficient condition ensuring the Aubin property of
solution maps to general implicitly defined multifunctions. This property itself has been introduced in
\cite{Au} and became gradually one of the most important stability notions for multifunctions. It is
widely used in {\em post-optimal analysis}, as a useful {\em qualification condition} in generalized
differentiation and it is closely connected with several important classical results like, e.g., the
theorems of Lyusternik and Graves \cite[pp. 275-276]{DR}.

This paper is focused on the Aubin property of {\em solution maps} to parameter-dependent variational
systems and extends the currently available results collected, e.g., in \cite{DR}.
 An efficient application of the new criterion in case of standard variational systems requires
  our ability to compute graphical derivatives and directional limiting coderivatives of
 normal-cone mappings to the considered constraint sets. Unfortunately, the calculus of directional
 limiting objects is not yet sufficiently developed and also in computation of graphical derivatives of
 normal-cone mappings one often meets various too restrictive assumptions. In this paper we will
 compute graphical derivatives and directional limiting coderivatives   of normal cone mappings
 associated with the sets $\Gamma$ of the form
\begin{equation}\label{eq-1}
\Gamma = g^{-1}(D)
\end{equation}
under reasonable assumptions imposed on the mapping $g$ and the set $D$.

 To this aim we will significantly improve the results from \cite{MOR1} and \cite{MOR2} concerning the
 graphical derivative and from \cite[Theorem 4.1]{MOR2} concerning the regular coderivative of the
 normal-cone mapping associated with \eqref{eq-1}. The resulting new second-order chain rules are valid
 under substantially relaxed reducibility and nondegeneracy assumptions compared with the preceding
 results of this type and are thus important for their own sake, not only in the context of this paper.
 Concretely, the new formula for the graphical derivative could be used, e.g., in testing the so-called
 isolated calmness of solution maps to variational systems (\cite{HKO}, \cite{MOR1}, \cite{MOR2}).

 The main result (Theorem \ref{ThMain}) represents a variant of \cite[Theorem 4.4]{GO3} tailored to a
 broad class of parameterized variational systems. It improves the sharpness of the currently available
 criteria for the Aubin property in the frequently arising case when the considered parametrization is not {\em ample}, cf.   \cite[Definition 1.1]{DR1}.

The plan of the paper is as follows. In Section 2 we summarize the needed notions from variational
analysis, state the main problem and recall \cite[Theorem 4.4]{GO3} which will be used as the main tool
in our development. Section 3 is devoted to the new results concerning the mentioned graphical
derivatives and directional limiting coderivatives of the normal-cone mapping related to $\Gamma$.  In
Section 4 we will formulate the resulting new criteria for the Aubin property of the considered
solution maps and illustrate their application by means of an example. It shows the ability of the
presented approach to deal with $\Gamma$ given by nonlinear programming (NLP) constraints.   Section 5
contains some amendments which may be useful for  genuine conic constraints. In particular, we consider
the case when $D$ amounts to the Carthesian product of Lorentz cones.  \\

Our notation is standard. For a set $A$, ${\rm lin} A$ denotes the {\em linearity space} of $A$, i.e.,
the largest linear space contained in $A$, ${\rm sp} A$ is the linear hull of $A$ and $P_{A}(\cdot)$
stands for the mapping of metric projection onto $A$. For a multifunction $F$, $\Gr F$ denotes its
graph and ${\rm rge} F$ denotes its range, i.e., ${\rm rge} F:=  \{y|y \in F(x) \mbox{ for  } x \in
{\rm dom} F \}$. For a cone $K, K^{\circ}$ is the (negative) polar cone, $\mathbb{B}, \mathbb{S}$ are
the unit ball and the unit sphere, respectively, and for a vector $a$, $[a]$ stands for the linear
subspace generated by $a$. Finally, $\stackrel{A}{\rightarrow}$ means the convergence within a set $A$.

\section{Problem formulation and preliminaries}

In the first part of this section we introduce some notions from variational analysis which will be
extensively used throughout the whole paper. Consider first a general closed-graph multifunction
$F:\mathbb{R}^{n} \rightrightarrows \mathbb{R}^{z}$ and its inverse $F^{-1}:\mathbb{R}^{z}
\rightrightarrows \mathbb{R}^{n}$ and assume that $(\bar{u},\bar{v})\in \Gr F$. 
\begin{definition}\label{DefAubCalm} We say that $F$ has the {\em Aubin property} around
$(\bar{u},\bar{v})$, provided there are neighborhoods $U$ of $\bar{u}$, $V$ of $\bar{v}$ and a modulus
$\kappa > 0$ such that
\[
F(u_{1}) \cap V \subset F(u_{2})+\kappa \| u_{1}-u_{2} \| \mathbb{B} \mbox{ for all } u_{1}, u_{2} \in
U.
\]
$F$ is said to be {\em calm} at $(\bar{u},\bar{v})$, provided there is a neighborhood $V$ of
$\bar{v}$ and a modulus $\kappa > 0$ such that
\[
F(u) \cap V \subset F(\bar{u})+\kappa \| u-\bar{u} \| \mathbb{B} \mbox{ for all } u \in
\mathbb{R}^{n}.
\]
\end{definition}
It is clear that the calmness is substantially weaker (less restrictive) than the Aubin property.
Furthermore, it is known that $F$ is calm at $(\bar{u},\bar{v})$ if and only if $F^{-1}$ is {\em
metrically subregular at} $(\bar{u},\bar{v})$, i.e., there is a neighborhood $V$ of $\bar{v}$ and a
modulus $\kappa > 0$ such that 
\[
d(v,F(\bar{u}))\leq \kappa d(\bar{u},F^{-1}(v)) \mbox{ for all } v \in V,
\]
cf. \cite[Exercise 3H.4]{DR}.

 To conduct a thorough analysis of the above stability notions one typically makes use of some basic
 notions of  generalized differentiation, whose definitions are presented below.

\begin{definition}\label{DefVarGeom} Let $A$  be a closed set in $\mathbb{R}^{n}$ and $\bar{x} \in A$.
\begin{enumerate}
 \item [(i)]
 \[
 T_{A}(\bar{x}):=\Limsup\limits_{t\searrow 0} \frac{A-\bar{x}}{t}
 \]
 is the {\em tangent (contingent, Bouligand) cone} to $A$ at $\bar{x}$ and
 \[
 \hat{N}_{A}(\bar{x}):=(T_{A}(\bar{x}))^{\circ}
 \]
 is the {\em regular (Fr\'{e}chet) normal cone} to $A$ at $\bar{x}$.
 \item [(ii)]
 \[
  N_{A}(\bar{x}):=\Limsup\limits_{\stackrel{A}{x \rightarrow \bar{x}}} \hat{N}_{A}(x)
 \]
 is the {\em limiting (Mordukhovich) normal cone} to $A$ at $\bar{x}$ and, given a direction $d \in
 \mathbb{R}^{n}$,
 \[
 N_{A}(\bar{x};d):= \Limsup\limits_{\stackrel{t\searrow 0}{d^{\prime}\rightarrow
 d}}\hat{N}_{A}(\bar{x}+ td^{\prime})
 \]
 is the {\em directional limiting normal cone} to $A$ at $\bar{x}$ {\em in direction} $d$ .
 \end{enumerate}
\end{definition}
The symbol ``Limsup''  stands for the outer (upper) set limit in the sense of Painlev\'{e}-Kuratowski,
cf. \cite[Chapter 4B]{RoWe98}. If $A$ is convex, then both the regular and the limiting normal cones
coincide with the classical normal cone in the sense of convex analysis. Therefore we will use in this
case the notation $N_{A}$.

By the definition, the limiting normal cone coincides with the directional limiting normal cone in direction
$0$, i.e.,  $N_A(\xb)=N_A(\xb;0)$, and $N_A(\xb;d)=\emptyset$ whenever $d\not\in T_A(\xb)$.

The above listed cones enable us to describe the local behavior of multifunctions via various
generalized derivatives. Consider again the multifunction $F$ and the point $(\bar{u},\bar{v})\in \Gr
F$.

\begin{definition}\label{DefGenDeriv}
\begin{enumerate}
\item [(i)]
The multifunction $D F(\bar{u},\bar{v}):\mathbb{R}^{n} \rightrightarrows\mathbb{R}^{z}$, defined by
\[
DF(\bar{u},\bar{v})(d):= \{h \in \mathbb{R}^{z}| (d,h)\in T_{\Gr F}(\bar{u},\bar{v})\}, d \in
\mathbb{R}^{n}
\]
is called the {\em graphical derivative} of $F$ at $(\bar{u},\bar{v})$;
\item[(ii)]
 The multifunction $\hat D^\ast F(\bar{u},\bar{v} ):
 \mathbb{R}^{z}\rightrightarrows\mathbb{R}^{n}$, defined by
\[
\hat D^\ast F(\bar{u},\bar{v} )(v^\ast):=\{u^\ast\in \mathbb{R}^{n} | (u^\ast,- v^\ast)\in \hat
N_{\Gr F}(\bar{u},\bar{v} )\}, v^\ast\in \mathbb{R}^{z}
\]
is called the {\em regular (Fr\'echet) coderivative} of $F$ at $(\bar{u},\bar{v} )$.
\item [(iii)]
 The multifunction $D^\ast F(\bar{u},\bar{v} ): \mathbb{R}^{z}\rightrightarrows\mathbb{R}^{n}$,
 defined by
\[
D^\ast F(\bar{u},\bar{v} )(v^\ast):=\{u^\ast\in \mathbb{R}^{n} | (u^\ast,- v^\ast)\in N_{\Gr
F}(\bar{u},\bar{v} )\}, v^\ast\in \mathbb{R}^{z}
\]
is called the {\em limiting (Mordukhovich) coderivative} of $F$ at $(\bar{u},\bar{v} )$.
\item [(iv)]
 Finally, given a pair of directions $(d,h) \in \mathbb{R}^{n} \times \mathbb{R}^{z}$, the
 multifunction \\
 $D^\ast F((\bar{u},\bar{v} ); (d,h)):
 \mathbb{R}^{n}\rightrightarrows\mathbb{R}^{z}$, defined by
\begin{equation}\label{eq-150}
D^\ast  F((\bar{u},\bar{v} ); (d,h))(v^\ast):=\{u^\ast \in \mathbb{R}^{n} | (u^\ast,-v^\ast)\in
N_{\Gr F}((\bar{u},\bar{v} ); (d,h)) \}, v^\ast\in \mathbb{R}^{z}
\end{equation}
is called the {\em directional limiting coderivative} of $F$ at $(\bar{u},\bar{v} )$ in
direction $(d,h)$.
\end{enumerate}
\end{definition}

For the properties of the cones $T_A(\xb)$, $\hat N_A(\xb)$ and $N_A(\xb)$ from Definition
\ref{DefVarGeom} and generalized derivatives (i), (ii) and (iii) from Definition \ref{DefGenDeriv} we
refer the interested reader to the monographs \cite{RoWe98} and \cite{M1}. The directional limiting normal cone and coderivative were introduced by the first author in \cite{Gfr13a} and various properties of these objects
can be found in \cite{GO3} and the references
therein. Note that $D^\ast  F((\bar{u},\bar{v} ))=D^\ast  F((\bar{u},\bar{v} ); (0,0))$ and that
$\dom D^\ast  F((\bar{u},\bar{v} ); (d,h))=\emptyset$ whenever $h\not\in DF(\bar{u},\bar{v})(d)$.

Let now  $M:\mathbb{R}^{l}\times\mathbb{R}^{n}\rightrightarrows \mathbb{R}^{m}$ be a given
multifunction with a closed graph and $S:\mathbb{R}^{l} \rightrightarrows \mathbb{R}^{n}$ be the
associated {\em implicit multifunction} given by 
\begin{equation}\label{eq-99}
S(p):=\{x \in \mathbb{R}^{m} | 0 \in M(p,x)\}.
\end{equation}
In what follows, $p$ will be called the {\em parameter} and $x$ will be the {\em decision variable}.
Given a {\em reference pair} $(\bar{p},\bar{x}) \in \Gr S$, one has the following criterion for the
Aubin property of $S$ around $(\bar{p},\bar{x})$.

\begin{theorem}\label{ThImplMapping}
(\cite[Theorem 4.4, Corollary 4.5]{GO3}). Assume that
\begin{enumerate}
 \item [(i)]
 \begin{equation}\label{eq-100}
\{u|0\in DM(\bar{p},\bar{x},0)(q,u)\}\neq \emptyset \mbox{ for all } q \in \mathbb{R}^{l};
 \end{equation}
 \item [(ii)]
 $M$ is metrically subregular at $(\bar{p},\bar{x},0)$;
 \item [(iii)]
 For every nonzero $(q,u)\in \mathbb{R}^{l} \times \mathbb{R}^{n}$ verifying $0 \in
 DM(\bar{p},\bar{x},0)(q,u)$ one has the implication
 \begin{equation}\label{eq-101}
(q^{*},0)\in D^{*}M ((\bar{p},\bar{x},0); (q,u,0))(v^{*})\Rightarrow q^{*}=0.
 \end{equation}
 \end{enumerate}
 Then $S$ has the Aubin property around $(\bar{p},\bar{x})$ and for any $q \in \mathbb{R}^{l}$
 \begin{equation}\label{eq-102}
DS(\bar{p},\bar{x})(q)=\{u|0\in DM(\bar{p},\bar{x},0)(q,u)\}.
 \end{equation}
 The above assertions remain true provided assumptions (ii), (iii) are replaced by
 \begin{enumerate}
\item [(iv)] For every nonzero $(q,u)\in \mathbb{R}^{l} \times \mathbb{R}^{n}$ verifying $0 \in DM
    (\bar{p},\bar{x},0)(q,u)$ one has the implication
\begin{equation}\label{eq-103}
(q^{*},0)\in D^{*}M ((\bar{p},\bar{x},0); (q,u,0))(v^{*})\Rightarrow
\left \{
\begin{array}{l}
q^{*}=0\\ v^{*}=0.
\end{array}
\right.
 \end{equation}
 \end{enumerate}
\end{theorem}

In this paper we will consider the case of {\em variational systems} where 
\begin{equation}\label{eq-104}
M(p,x):=H(p,x)+\hat{N}_{\Gamma}(x), ~\Gamma =g^{-1}(D) .
\end{equation}
 In (\ref{eq-104}), $H:\mathbb{R}^{l}\times \mathbb{R}^{n} \rightarrow \mathbb{R}^{n}$ is
continuously differentiable, $g:\mathbb{R}^{n} \rightarrow \mathbb{R}^{s}$ is twice continuously
differentiable and $D \subset \mathbb{R}^{s}$ is a closed  set.

We recall from \cite{GO3} that Theorem \ref{ThImplMapping} provides us in case of $M$ given by
\eqref{eq-104} with sharper (more restrictive) sufficient conditions than the currently available
criteria whenever $\nabla_p H(\pb,\xb)$ is not surjective, i.e., the considered parameterization  is not
ample at $(\pb,\xb)$.

By the continuous differentiability of $H$ one has that for $M$ given in (\ref{eq-104}) and any
$(q,u)\in \mathbb{R}^{l} \times \mathbb{R}^{n}$ 
\begin{equation}\label{eq-106}
\begin{split}
& DM(\bar{p},\bar{x},0)(q,u)= \\
  &
  \nabla_{p}H(\bar{p},\bar{x})q+\nabla_{x}H(\bar{p},\bar{x})u+D\hat{N}_{\Gamma}(\bar{x},-H(\bar{p},\bar{x}))
(u,-\nabla_{p}H(\bar{p},\bar{x})q-\nabla_{x}H(\bar{p},\bar{x})u),
\end{split}
\end{equation}
cf. \cite[Exercise 10.43]{RoWe98}. Likewise, for any $v^{*}\in \mathbb{R}^{n}$,
\begin{equation}\label{eq-107}
\begin{array}{l}
D^{*}M((\bar{p},\bar{x},0); (q,u,0))(v^{*})=\\[1ex] 
\left[ \begin{array}{l}
\nabla_{p}H(\bar{p},\bar{x})^{T}v^{*}\\
\nabla_{p}H(\bar{x},\bar{x})^{T}v^{*}+D^{*}\hat{N}_{\Gamma}((\bar{x},-H(\bar{p},\bar{x}));(u,-\nabla_{p}H(\bar{p},\bar{x})q-
\nabla_{x}H(\bar{p},\bar{x})u))(v^{*})
\end{array}\right],
\end{array}
\end{equation}
cf. \cite[Theorem 2.10]{GO3}. The application of Theorem \ref{ThImplMapping} requires thus the
computation of $D\hat{N}_{\Gamma}(\bar{x},-H(\bar{p},\bar{x}))(\cdot, \cdot)$
 and $D^{*}\hat{N}_{\Gamma}((\bar{x},-H(\bar{p},\bar{x})); (\cdot , \cdot))(v^{*})$ for directions
 generated by the vectors $q,u$. This problem will be tackled in the next section.

\section{Graphical derivatives and directional limiting coderivatives of $\hat{N}_{\Gamma}$}

Throughout this section we will impose a  weakened version of the {\em reducibility} and the
{\em nondegeneracy} conditions introduced in \cite{BS}. Concretely, in what follows we will assume that
\begin{enumerate}
 \item [(A1):]
 There exists a closed set $\Theta \subset \mathbb{R}^{d}$ along with a twice continuously
 differentiable mapping $h:\mathbb{R}^{s}\rightarrow \mathbb{R}^{d}$ and a neighborhood
 $\mathcal{V}$ of $g(\bar{x})$ such that $\nabla h(g(\bar{x}))$ is surjective and
 \[
 D \cap \mathcal{V}= \{z \in \mathcal{V}|h(z)\in \Theta\};
 \]
 \item [(A2):]
 \begin{equation}\label{eq-3.1}
\range \nabla g (\bar{x})+ \ker \nabla h (g(\bar{x}))=\mathbb{R}^{l}.
 \end{equation}
 \end{enumerate}
 Note that conditions (A1), (A2) amount to the reducibility of $D$ to $\Theta$ at $g(\bar{x})$ and the
 nondegeneracy of $\bar{x}$ with respect to $\Gamma$ and the mapping $h$ in the sense of \cite{BS}
 provided the sets $D,\Theta$ are convex. The assumptions (A1), (A2) have the following important
 impact on the representation of $\Gamma$ and $\hat{N}_{\Gamma}$ near $\bar{x}$.

\bigskip

\begin{proposition}\label{PropBasic}
Let $b:=h \circ g$. Then there exists  neighborhoods $\mathcal{U}$ of $\bar{x}$ and $\mathcal{W}
\supset g(\mathcal{U})$ of $g(\bar{x})$ such that 
\begin{equation}\label{eq-3.2}
\Gamma \cap \mathcal{U}= \{x \in \mathcal{U}| b(x)\in \Theta\},
\end{equation}
$\nabla b (x)$ is surjective for every $x \in \mathcal{U}, \nabla h(y)$ is surjective for every
 $y \in \mathcal{W}$ and
\begin{align}
&\hat{N}_{D}(y)=\nabla h(y)^{T}\hat{N}_{\Theta}(h(y)), y \in \mathcal{W},\label{eq-3.3}\\ &\hat{N}_{\Gamma}(x)=
\nabla b(x)^{T}\hat{N}_{\Theta}(b(x)) = \nabla g(x)^{T}\hat{N}_{D}(g(x)), x \in
\mathcal{U}.\label{eq-3.101}
\end{align}
\end{proposition}

 \begin{proof}

 First we show that \eqref{eq-3.1} is equivalent with the surjectivity of $\nabla b(\xb) = \newline
 \nabla
 h(g(\xb))\nabla g(\xb)$. Indeed, $\nabla b(\xb)$ is surjective if and only if
  \[\{0\}=\ker \nabla b(\xb)^T=\ker (\nabla g(\xb)^T\nabla h(g(\xb))^T),\]
  which, by the assumed surjectivity of $\nabla h(g(\xb))$, in turn holds if and only if
  \begin{eqnarray*}\{0\}&=& \ker \nabla
  g(\xb)^T\cap ~ \range \nabla h(g(\xb))^T
   =\Big((\ker \nabla g(\xb)^T)^\perp+ (\range \nabla h(g(\xb))^T)^\perp\Big)^\perp\\
   &=&(\range \nabla
   g(\xb)+\ker \nabla h(g(\xb)))^\perp\end{eqnarray*}
  and this is clearly equivalent with \eqref{eq-3.1}.
  Hence $\nabla b(\xb)$ is surjective and we can find  open neighborhoods $\mathcal{W}\subset
  \mathcal{V}$ and $\mathcal{U}\subset g^{-1}(\mathcal{W})$ of $\xb$ such that $\nabla b(x)$ is
  surjective for all $x\in \mathcal{U}$ and $\nabla h(y)$ is surjective for all $y\in
  \mathcal{W}$, where $\mathcal{V}$ is given by assumption (A1). Hence for every $x\in \mathcal{U}$ we
  have $g(x)\in \mathcal{V}$ and (\ref{eq-3.2}) follows from (A1). The descriptions of the regular
  normal cones (\ref{eq-3.3}),
  (\ref{eq-3.101}) result from \cite[Exercise 6.7]{RoWe98}.
  \qed\end{proof}

\begin{remark}
  Note that, given a vector $x^\ast\in\widehat N_\Gamma(x)$ with $x\in \Gamma\cap \mathcal{U}$, there
  is a  unique $\lambda\in N_D(g(x))$ satisfying
  \begin{equation}
     \label{Eqxast}x^\ast=\nabla g(x)^T\lambda.
  \end{equation}
  Indeed, from (\ref{eq-3.101}) it follows that there is a unique $\mu\in\widehat N_{\Theta}(b(x))$
  such that $x^\ast=\nabla b(x)^T\mu$ thanks to the surjectivity of $\nabla b(x)$. Since
  $\lambda=\nabla
  h(g(x))^T\mu$, we are done.
\end{remark}

The rest of this section is divided to two subsections devoted to the graphical derivatives and the
directional limiting coderivatives of $\hat{N}_{\Gamma}$, respectively.

\subsection{Graphical derivatives of $\hat{N}_{\Gamma}$}

The computation of graphical derivatives of $\hat{N}_{\Gamma}$ has been considered in  numerous works,
see \cite{RoWe98} and the references therein. Recently, in \cite{MOR1} and \cite{MOR2} the authors have
derived two different
 formulas for $D \hat{N}_{\Gamma}$ by using a strengthened variant of (A1), (A2) together with some
 additional assumptions. They include either the convexity of $\Gamma$ or a special {\em projection
 derivation condition} (PDC) defined next.
\begin{definition}\label{PDC}
 A convex set $\Xi$ satisfies the {\em projection derivation condition} (PDC) at the point $\bar{z} \in \Xi$ if we have
 \[
 P_{\Xi}(\bar{z} + b; h)=P_{K(\bar{z},b)} (h) ~\mbox{ for all }~ b \in N_{\Xi}(\bar{z})~\mbox{ and }~ h \in \mathbb{R}^{s},
 \]
where $K (\bar{z},b):=T_{\Xi}(\bar{z}) \cap \{b\}^{\perp}$.
 \end{definition}

 In our case the PDC condition is automatically fulfilled provided $D$ is convex polyhedral.
 Throughout sections 3.1. and 3.2  it is enough to assume,  however,  the weakened reducibility and nondegeneracy assumptions (A1), (A2) and we obtain
 new workable formulas without any additional requirements.

 \bigskip

\begin{theorem}\label{ThGraphDer}
 Let assumptions (A1), (A2) be fulfilled, $\bar{x}^{*}\in \hat{N}_{\Gamma}(\bar{x})$ and
 $\bar{\lambda}$  be the (unique) multiplier satisfying
 \begin{equation}\label{eq-3.55} \lb\in  \hat{N}_{D}(g(\bar{x})),\;
\nabla g(\xb)^T\lb=\xba.
 \end{equation}
  Then
 \begin{equation}\label{eq-3.6}
\begin{split}
&T_{\Gr \hat{N}_{\Gamma}} (\bar{x},\bar{x}^{*})= \\ &\{(u, u^{*})| \exists \xi : (\nabla g (\bar{x})u,
\xi)\in T_{\Gr \hat{N}_{D}} (g(\bar{x}), \bar{\lambda}), u^{*}=\nabla g (\bar{x})^{T}\xi +
\nabla^{2} \skalp{ \bar{\lambda},g} (\bar{x})u \}.
\end{split}
 \end{equation}
\end{theorem}
  \begin{proof}
Let $(u,u^\ast)\in T_{\Gr \widehat N_\Gamma}(\xb,\xba)$ and consider sequences $t_k\searrow 0$ and
$(u_k,u_k^\ast)\to (u,u^\ast)$ with $x_k^\ast:=\xba+t_ku_k^\ast\in \widehat N_{\Gamma}(x_k)$, where
$x_k:=\xb+t_ku_k$. We can assume that $x_k\in \mathcal{U}$  and that $\nabla b(x_k)$ is surjective for
all $k$, where $b$ and $\mathcal{U}$ are given by Proposition \ref{PropBasic}. Hence we can find
multipliers  $\mu^k\in \widehat N_\Theta(b(x_k))$ such that $x_k^\ast=\nabla b(x_k)^T\mu^k$. The
sequence $\mu^k$ is bounded and, after passing to some subsequence, converges to some $\bar\mu\in
\widehat N_\Theta(h(g(\xb)))$ with $\xba=\nabla b(\xb)^T\bar\mu$. Further, by (\ref{eq-3.3}) we have
$\bar\lambda=\nabla h(g(\xb))^T\mu$ for some $\mu\in \widehat N_\Theta(h(g(\xb)))$ implying
$\xba=\nabla b(\xb)^T\mu$ and $\bar\mu=\mu$ follows from the surjectivity of $\nabla b(\xb)$.\\ 
Since
\[t_ku_k^\ast=x_k^\ast-\xba=\nabla b(x_k)^T\mu_k-\nabla b(\xb)^T\bar \mu
=t_k\nabla^2\skalp{\bar\mu, b}(\xb)u_k+\nabla b(\xb)^T(\mu^k-\bar\mu)+\oo(t_k),\] we obtain that
\[\nabla b(\xb)^T\frac{\mu^k-\bar\mu}{t_k}=u^\ast-\nabla^2\skalp{\bar\mu, b}(\xb)u+\oo(t_k)/t_k.\]
By the surjectivity of $\nabla b(\xb)$ we obtain that the sequence $\eta^k:=(\mu^k-\bar \mu)/t_k$ is
bounded and, after passing to some subsequence, $\eta^k$ converges to some $\eta$ fulfilling
\[\nabla b(\xb)^T\eta=u^\ast-\nabla^2\skalp{\bar\mu, b}(\xb)u.\]
Denoting $\lambda^k=\nabla h(g(x_k))\mu^k$ we obtain $\lambda^k\in \widehat N_D(g(x_k))$ by
(\ref{eq-3.3}) and
\[
\begin{split}
& \lambda^k-\bar\lambda=\\
& \nabla h(g(x_k))^T\mu_k-\nabla h(g(\xb))^T\bar\mu=\nabla^2\skalp{\bar\mu,
h}(g(\xb))\nabla g(\xb)(t_ku_k)+\nabla h(g(\xb))^T(\mu_k-\bar\mu)+\oo(t_k),
\end{split}
\]
implying that $(\lambda^k-\bar\lambda)/t_k$ converges to
\begin{equation}\label{EqXi}\xi:=\nabla^2\skalp{\bar\mu, h}(g(\xb))\nabla g(\xb)u+\nabla
h(g(\xb))^T\eta.
\end{equation}
 We conclude $(\nabla g(\xb)u,\xi)\in T_{\Gr \widehat N_D}(g(\xb),\bar\lambda)$ and
\begin{eqnarray*}u^\ast&=&\nabla b(\xb)^T\eta+\nabla^2\skalp{\bar\mu, b}(\xb)u\\
&=&\nabla g(\xb)^T\nabla h(g(\xb))^T\eta+\nabla g(\xb)^T\nabla^2\skalp{\bar\mu, h}(g(\xb))\nabla
g(\xb)u+\nabla^2\skalp{\nabla h(g(\xb)^T\bar\mu, g}(\xb)u\\ &=&\nabla
g(\xb)^T\xi+\nabla^2\skalp{\lb,g}(\xb)u\end{eqnarray*} showing
\[
\begin{split}
&(u,u^\ast)\in\\
& {\cal T}:=\{(u,u^\ast)\mv \exists \xi: (\nabla g(\xb)u,\xi)\in T_{\Gr\widehat
N_D}(g(\xb),\lambda),\ u^\ast= \nabla g(\xb)^T\xi+\nabla^2\skalp{\lb,g}(\xb)u\}
\end{split}
\]
Thus $T_{\Gr\widehat N_\Gamma}(\xb,\xba)\subset {\cal T}$ holds.

In order to show the reverse inclusion $T_{\Gr\widehat N_\Gamma}(\xb,\xba)\supset {\cal T}$, consider
$(u,u^\ast)\in {\cal T}$ together with some corresponding $\xi$. Then there are sequences $t_k\searrow
0$, $v_k\to\nabla g(\xb)u$ and $\xi^k\to\xi$ such that $\bar\lambda+t_k\xi^k\in\widehat
N_D(g(\xb)+t_kv_k)$ and thus $h(g(\xb)+t_kv_k)\in \Theta$ and $\bar\lambda+t_k\xi^k=\nabla
h(g(\xb)+t_kv_k))^T\mu^k$ with $\mu^k\in \widehat N_\Theta(h(g(\xb)+t_kv_k)$ for all $k$ sufficiently
large. Further, the sequence $\mu^k$ is bounded. Since
\begin{eqnarray*}b(\xb+t_ku_k)-h(g(\xb)+t_kv_k)&=&\nabla b(\xb)(t_ku_k)-\nabla
h(g(\xb))(t_kv_k)+\oo(t_k)\\
&=&t_k\nabla h(g(\xb))(\nabla g(\xb)u-v_k)+\oo(t_k)=\oo(t_k)\end{eqnarray*} and $\nabla b(\xb)$ is
surjective, we can find for each $k$ sufficiently large some $x_k$ with
$b(x_k)=h(g(\xb)+t_kv_k)\in\Theta$ and $x_k-(\xb+t_ku_k)=\oo(t_k)$. It follows that
\[\nabla b(x_k)^T\mu^k=\nabla g(x_k)^T\nabla h(g(x_k))^T\mu^k\in\widehat N_\Gamma(x_k)\]
and
\begin{eqnarray*}\nabla b(x_k)^T\mu^k-\xba&=&\nabla g(x_k)^T\nabla h(g(x_k))^T\mu^k-\xba\\
&=&\nabla g(x_k)^T\big(\nabla h(g(x_k))^T\mu^k-\bar
\lambda\big)+\nabla^2\skalp{\lb,g}(\xb)(x_k-\xb)+\oo(t_k)\\ &=&\nabla g(x_k)^T\big(\nabla
h(g(\xb)+t_kv_k)^T\mu^k-\bar
\lambda+\oo(t_k)\big)+t_k\nabla^2\skalp{\lb,g}(\xb)u+\oo(t_k)\\ &=&t_k\nabla
g(x_k)^T\xi^k+t_k\nabla^2\skalp{\lb,g}(\xb)u+\oo(t_k)\\ &=&t_k\big(\nabla
g(\xb)^T\xi+\nabla^2\skalp{\lb,g}(\xb)u\big)+\oo(t_k)\end{eqnarray*} showing $(u,u^\ast)\in T_{\Gr
\widehat N_\Gamma}(\xb,\xba)$.
\qed\end{proof}
 \begin{remark}
  Everything remains true if we replace $\widehat N_\Gamma$, $\widehat N_D$, $\widehat N_\Theta$ by
  $N_\Gamma$, $N_D$, $N_\Theta$,
\end{remark}
\begin{remark}
  Note that to each pair $(u,u^\ast)\in T_{\Gr\widehat N_\Gamma}(\xb,\xba)$ there is a unique $\xi$
  satisfying the relations on the right-hand side of  (\ref{eq-3.6}). Its existence has been shown in
  the first
  part of the proof and its uniqueness follows from \eqref{EqXi} and the uniqueness of $\eta$ implied
  by the surjectivity  of $\nabla b(\xb)$.
\end{remark}
 From (\ref{eq-3.6}) one can relatively easily derive the formulas from \cite{MOR1} and \cite{MOR2} by
 imposing appropriate additional assumptions. Indeed, let us suppose that, in addition to (A1), (A2),
 $D$  is convex and the (single-valued) operator $P_{D}$ is directionally differentiable at
 $g(\bar{x})$. Then one has the relationship
 \[
 \begin{array}{l}
 T_{\Gr N_{D}}(g(\bar{x}), \bar{\lambda})=
 \left \{
 (v,w) \left |
 \left[ \begin{array}{l}
 v + w\\
 v
 \end{array}\right]
 \in T_{\Gr P_{D}}(g(\bar{x})+\bar{\lambda}, g(\bar{x}))
 \right. \right\} =\\
  \{(v,w)| v=P^{\prime}_{D}(g(\bar{x})+\bar{\lambda}; v+w)\},
 \end{array}
 \]
 which implies that under the posed additional assumptions the relation
 \begin{equation}\label{eq-3.65}
(\nabla g(\bar{x})u, \xi)\in T_{\Gr N_{D}}(g(\bar{x}), \bar{\lambda})
 \end{equation}
 amounts to the equation
 \begin{equation}\label{eq-3.7}
\nabla g(\bar{x})u = P^{\prime}_{D} (g(\bar{x})+\bar{\lambda}; \nabla g(\bar{x})u+\xi).
 \end{equation}
 Formula (\ref{eq-3.6}) attains thus exactly the form from \cite[Theorem 3.3]{MOR1}. Note that in this
 way it was not necessary to assume the convexity of $\Gamma$ like in \cite{MOR1}. Thanks to this, upon
 imposing the PDC condition on $D$ at $g(\bar{x})$, one gets from (\ref{eq-3.7})
 that
 \begin{equation}\label{eq-3.8}
\nabla g(\bar{x})u = P_{K}(\nabla g(\bar{x})u+\xi),
 \end{equation}
 where $K$ stands for the critical cone to $D$ at $g(\bar{x})$ with respect to $\bar{\lambda}$, i..e.,
 $K = T_{D}(g(\bar{x}))\cap [\bar{\lambda}]^{\perp}$. From (\ref{eq-3.8}) we easily deduce that
 \[
 \xi \in N_{K}(\nabla g(\bar{x})u)
 \]
 and relation (\ref{eq-3.6}) thus simplifies to
 \begin{equation}\label{eq-3.9}
T_{\Gr \hat{N}_{\Gamma}} (\bar{x}, \bar{x}^{*})=\{(u,u^{*})|u^{*}\in \nabla^{2} \langle \bar{\lambda},g
\rangle(\bar{x})u+\nabla g(\bar{x})^{T}N_{K}(\nabla g(\bar{x})u) \}.
 \end{equation}
 We have recovered the formula from \cite[Theorem 5.2]{MOR2}. This enormous simplification of the way
 how this result has been derived is due to Theorem \ref{ThGraphDer} and the
 equivalence of relations (\ref{eq-3.65}), (\ref{eq-3.7}) (under the posed additional assumptions).

 As mentioned above, the PDC condition automatically holds whenever $D$ is a convex polyhedral set. Thus, for instance, in
 case of variational systems with $\Gamma$ given by NLP constraints, one can compute
 $DM(\bar{p},\bar{x},0)(q,u)$ by the workable formula
 \begin{equation}\label{eq-3.10}
DM(\bar{p},\bar{x},0)(q,u)=\nabla_{p}H(\bar{p},\bar{x})q+\nabla_{x}\mathcal{L}(\bar{p},\bar{x},\bar{\lambda})u+\nabla
g(\bar{x})^{T}N_{K}(\nabla g(\bar{x})u),
 \end{equation}
 where
 \[
 \mathcal{L}(p,x,\lambda):= H(p,x)+\nabla g(x)^{T}\lambda
 \]
 is the {\em Lagrangian} associated with the considered variational system.

 \subsection{Regular and directional limiting coderivatives of $\hat N_\Gamma$}

 \begin{theorem}\label{ThRegNormalCone}
   Let assumptions (A1), (A2) be fulfilled, $\xba\in\hat N_\Gamma(\xb)$ and $\lb$ be the (unique)
   multiplier satisfying \eqref{eq-3.55}. Then
   \begin{equation}\label{EqRegNormalCone}
   \hat N_{\Gr \hat N_\Gamma}(\xb,\xba)=\left\{(w^\ast,w)\mv \exists v^\ast:\begin{array}{l} (v^\ast,
   \nabla g(\xb)w)\in \hat N_{\Gr \hat N_D}(g(\xb),\lb),\\
    w^\ast= -\nabla^2\skalp{\lb,g}(\xb)w +\nabla g(\xb)^Tv^\ast\end{array}\right\}.
   \end{equation}
 \end{theorem}
 \begin{proof}
   First we justify \eqref{EqRegNormalCone} in the case when the derivative operator $\nabla
   g(\xb):\R^n\to\R^s$ is surjective. By the definition we have $(w^\ast,w)\in \hat N_{\Gr \hat
   N_\Gamma}(\xb,\xba)$ if and only if $\skalp{w^\ast,u}+\skalp{w,u^\ast}\leq 0$ $\forall (u,u^\ast)\in
   T_{\Gr \hat N_\Gamma}(\xb,\xba)$, which by virtue of Theorem \ref{ThGraphDer} is equivalent to the
   statement that $(0,0)$ is a global solution of the problem
   \begin{eqnarray*}
     \max_{u,\xi} &&\gamma(u,\xi):=\skalp{w^\ast,u}+\skalp{w,\nabla
     g(\xb)^T\xi+\nabla^2\skalp{\lb,g}(\xb)u}\\
     \mbox{subject to}&& (\nabla g(\xb)u,\xi)\in T_{\Gr \hat N_D}(g(\xb),\lb).
   \end{eqnarray*}
   Since the objective can be rewritten as
   $\gamma(u,\xi)=\skalp{w^\ast+\nabla^2\skalp{\lb,g}(\xb)w,u}+\skalp{\nabla g(\xb)w,\xi}$, this is in
   turn equivalent to the statement
   \[(w^\ast+\nabla^2\skalp{\lb,g}(\xb)w, \nabla g(\xb)w)\in C^\circ\]
   where $C:=\{(u,\xi)\mv (\nabla g(\xb)u,\xi)\in T_{\Gr \hat N_D}(g(\xb),\lb)\}$.
   By surjectivity of $\nabla g(\xb)$  the linear mapping $(u,\xi)\to (\nabla g(\xb)u,\xi)$ is
   surjective as well and we can apply \cite[Exercise 6.7]{RoWe98} to obtain
   \begin{eqnarray*}C^\circ&=&\hat N_C(0,0)=\{(\nabla g(\xb)^Tv^\ast,v)\mv(v^\ast,v)\in \hat N_{T_{\Gr
   \hat N_D}(g(\xb),\lb)}(0,0)\}\\
   &=&\{(\nabla g(\xb)^Tv^\ast,v)\mv(v^\ast,v)\in \hat N_{\Gr \hat N_D}(g(\xb),\lb)\}.\end{eqnarray*}
   Now formula \eqref{EqRegNormalCone} follows.

   It remains to replace the surjectivity of $\nabla g(\xb)$ by the weaker nondegeneracy assumption
   from (A2). To proceed, we employ the local representation of $D$ provided
    by its reducibility at $g(\xb)$, see assumption (A1). By Proposition \ref{PropBasic} we have
    $\Gamma\cap {\mathcal U}=\{x\in{\mathcal U}\mv b(x)\in \Theta\}$ and by assumption (A1) we have
    $D\cap{\mathcal V}=\{z\in {\mathcal V}\mv h(z)\in\Theta\}$, where $\mathcal U$ and $\mathcal V$
    denote  neighborhoods of $\xb$ and $g(\xb)$, respectively. Since both $\nabla b(\xb)$ and $\nabla
    h(g(\xb))$ are surjective, we can apply \eqref{EqRegNormalCone} twice to obtain
    \begin{equation}\label{EqNormalAux1}\hat N_{\Gr \hat N_\Gamma}(\xb,\xba)=\left\{(w^\ast,w)\mv
    \exists z^\ast: \begin{array}{l}(z^\ast, \nabla b(\xb)w)\in \hat N_{\Gr \hat
    N_\Theta}(b(\xb),\bar\mu),\\
     w^\ast= -\nabla^2\skalp{\bar\mu,b}(\xb)w +\nabla b(\xb)^Tz^\ast\end{array}\right\}\end{equation}
    and
    \begin{equation}\label{EqNormalAux2}\hat N_{\Gr \hat N_D}(g(\xb),\lb)=\left\{(v^\ast,v)\mv \exists
    z^\ast: \begin{array}{l}(z^\ast, \nabla h(g(\xb))v)\in \hat N_{\Gr \hat
    N_\Theta}(h(g(\xb)),\bar\mu),\\
     v^\ast= -\nabla^2\skalp{\bar\mu,h}(g(\xb))v +\nabla h(g(\xb))^Tz^\ast\end{array}\right\},
    \end{equation}
    where $\bar\mu$ is the unique multiplier satisfying $\lb=\nabla h(g(\xb))^T\bar \mu$. By the
    classical chain rule  we have $\nabla b(\xb)=\nabla h(g(\xb))\nabla g(\xb)$ and
    \begin{eqnarray*}\nabla^2\skalp{\bar\mu,b}(\xb)&=&\nabla
    g(\xb)^T\nabla^2\skalp{\bar\mu,h}(g(\xb))\nabla g(\xb) +\nabla^2\skalp{ \nabla h(g(\xb))^T\bar\mu,
    g}(\xb)\\
    &=&\nabla g(\xb)^T\nabla^2\skalp{\bar\mu,h}(g(\xb))\nabla g(\xb)+\nabla^2\skalp{\lb, g}(\xb).
    \end{eqnarray*}
    Now consider $(w^\ast,w)\in \hat N_{\Gr \hat N_\Gamma}(\xb,\xba)$ and let $z^\ast$ be chosen such
    that $(z^\ast, \nabla b(\xb)w)\in \hat N_{\Gr \hat N_\Theta}(b(\xb),\bar\mu)$ and $w^\ast=
    -\nabla^2\skalp{\bar\mu,b}(\xb)w +\nabla b(\xb)^Tz^\ast$. By substituting $v:=\nabla g(\xb) w$,
    $v^\ast:=-\nabla^2\skalp{\bar\mu,h}(g(\xb))q+\nabla h(g(\xb))^Tz^\ast$ we obtain $(z^\ast, \nabla
    h(g(\xb))v)\in \hat N_{\Gr \hat N_\Theta}(h(g(\xb)),\bar\mu)$ implying $(v^\ast,v)=(v^\ast, \nabla
    g(\xb) w )\in \hat N_{\Gr \hat N_D}(g(\xb),\lb)$ by \eqref{EqNormalAux2}
    and
    \begin{eqnarray*}w^\ast&=&-\nabla^2\skalp{\lb, g}(\xb)w+ \nabla
    g(\xb)^T\big(-\nabla^2\skalp{\bar\mu,h}(g(\xb))\nabla g(\xb)w+\nabla h(g(\xb))^Tz^\ast\big)\\
    &=&-\nabla^2\skalp{\lb, g}(\xb)w+ \nabla g(\xb)^Tv^\ast.
    \end{eqnarray*}
    Thus
    \[(w^\ast,w)\in {\cal N}:=\left\{(w^\ast,w)\mv \exists v^\ast:\begin{array}{l} (v^\ast, \nabla
    g(\xb)w)\in \hat N_{\Gr \hat N_D}(g(\xb),\lb),\\
    w^\ast= -\nabla^2\skalp{\lb,g}(\xb)w +\nabla g(\xb)^Tv^\ast\end{array}\right\}\]
    establishing the inclusion $\hat N_{\Gr \hat N_\Gamma}(\xb,\xba)\subset {\cal N}$. To establish the
    reverse inclusion consider $(w^\ast,w)\in{\cal N}$ together with the corresponding element
    $v^\ast$. By \eqref{EqNormalAux2} we can find some  $z^\ast$ such that $(z^\ast, \nabla
    h(g(\xb))\nabla g(\xb)w)=(z^\ast, \nabla b(\xb)w)\in \hat N_{\Gr \hat N_\Theta}(h(g(\xb)),\bar\mu)$
    and $v^\ast= -\nabla^2\skalp{\bar\mu,h}(g(\xb))\nabla g(\xb)w +\nabla h(g(\xb))^Tz^\ast$. Hence
    \begin{eqnarray*}w^\ast&=& -\nabla^2\skalp{\lb, g}(\xb)w+ \nabla g(\xb)^Tv^\ast\\
    &=&-(\nabla^2\skalp{\lb, g}(\xb)+\nabla g(\xb)^T\nabla^2\skalp{\bar\mu,h}(g(\xb))\nabla
    g(\xb))w+\nabla g(\xb)^T\nabla h(g(\xb))^Tz^\ast\\
    &=&-\nabla^2\skalp{\bar\mu,b}(\xb)w+\nabla b(\xb)^Tz^\ast
    \end{eqnarray*}
    and we conclude $(w^\ast,w)\in \hat N_{\Gr \hat N_\Gamma}(\xb,\xba)$ by \eqref{EqNormalAux1}. Hence
    $\hat N_{\Gr \hat N_\Gamma}(\xb,\xba)= {\cal N}$ and this finishes the proof.
 \qed\end{proof}

 By the definition of the regular coderivative we obtain the following Corollary.
 \begin{corollary}\label{CorRegCoderiv} Under the assumptions of Theorem \ref{ThRegNormalCone} one has
   \begin{equation}\label{EqRegCoderiv}
   \hat D^\ast \hat N_\Gamma(\xb,\xba)(w)=\nabla^2\skalp{\lb,g}(\xb)w+ \nabla g(\xb)^T \hat D^\ast\hat
   N_D(g(\xb),\lb)(\nabla g(\xb)w),\ w\in\R^n.
   \end{equation}
 \end{corollary}

 In order to show the following result on the directional limiting coderivative note that assumptions
 (A1) and (A2) hold for all $x\in\Gamma$ near $\xb$. In fact, by taking into account Proposition
 \ref{PropBasic} and its proof, we have that $\nabla h(g(x)$ and $\nabla b(x)$ are surjective for all
 $x$ near $\xb$ and the latter is equivalent with validity of the condition $\range \nabla g(x)+\ker
 \nabla h(g(x))=\R^n$ for those $x$.

\begin{theorem}\label{ThDirLimCoderiv}
 Let assumptions (A1), (A2) be fulfilled, $\bar{x}^{*}\in \hat{N}_{\Gamma}(\bar{x})$ and
 $\bar{\lambda}$ be the (unique) multiplier satisfying (\ref{eq-3.55}). Further we are given a pair of
 directions $(u,u^{*})\in \newline
  T_{\Gr\hat N_\Gamma}(\xb,\xba)$. Then for any $w\in \mathbb{R}^{n}$
\begin{eqnarray}\label{eq-3.13}
\lefteqn{D^{*}\hat{N}_{\Gamma}((\xb,\xba); (u,u^{*}))(w)}\\
\nonumber& = & \nabla^{2}\skalp{\lb, g} (x)w +
  \nabla g(\bar{x})^{T}D^{*}\hat{N}_{D}((g(\bar{x}),\bar{\lambda}); (\nabla g(\bar{x})u, \bar{\xi}))(\nabla g
  (\bar{x})w),
\end{eqnarray}
where $\bar{\xi}\in \mathbb{R}^{s}$ is the (unique) vector satisfying the relations
\begin{equation}\label{eq-3.14}
(\nabla g(\bar{x})u, \bar{\xi})\in T_{\Gr \hat{N}_{D}}(g(\bar{x}), \bar{\lambda}),\; u^{*}=
\nabla g (\bar{x})^{T}\bar{\xi}+\nabla^{2}\skalp{\bar{\lambda},g} (\bar{x})u.
\end{equation}
\end{theorem}
\begin{proof}
In the first step we observe that for arbitrary sequences
 $
 \vartheta_{k}\searrow 0, u_{k}\rightarrow u, u^{*}_{k}\rightarrow  u^{*} \mbox{ and } w_k\rightarrow w
 \mbox{ such that } (x_{k}, x^{*}_{k}):=
(\bar{x}+\vartheta_{k}u_{k}, \bar{x}^{*}_{k}+ \vartheta_{k} u^{*}_{k}) \in \Gr \hat{N}_{\Gamma}$ and
$k$ sufficiently large one has 
\[
\hat{D}^{*}\hat{N}_{\Gamma} (x_{k},x^{*}_{k})(w_k)=\nabla^{2}\skalp{\lambda_{k},g}(x_{k})w_k+
\nabla g(x_{k})^{T}\hat{D}^{*}\hat{N}_{D}(g(x_{k}),\lambda_{k})(\nabla g(x_{k})w_k),
\]
where $\lambda_{k}$ is the (unique) multiplier satisfying the relations
\begin{equation}\label{eq-3.15}
\nabla g(x_{k})^{T}\lambda_{k} = x^{*}_{k}, \,\, \lambda_{k}\in \hat{N}_{D}(g(x_{k})).
\end{equation}
Indeed, this follows immediately from Corollary \ref{CorRegCoderiv} due to the mentioned robustness
of assumptions (A1), (A2). Moreover, we know that $\lambda_{k} \rightarrow \bar{\lambda}$ which is the
unique multiplier satisfying \eqref{eq-3.55}.
 Next we observe that
\[g(x_{k})=g(\bar{x}) +\vartheta_{k}h_{k} \mbox{ with } h_{k}=\frac{g(x_{k})-g(\bar{x})}{\vartheta_{k}}
\rightarrow \nabla g(\bar{x})u\]
 and\\
$$\lambda_{k}= \bar{\lambda}+\vartheta_{k}\xi_{k} \mbox{ with } \xi_{k}=
\frac{\lambda_{k}-\bar{\lambda}}{\vartheta_{k}}.$$
 It follows that
\begin{eqnarray}\label{eq-3.17}
\lefteqn{\hat{D}^{*}\hat{N}_{\Gamma}(\bar{x}+\vartheta_{k} u_{k}, \bar{x}^{*}+\vartheta_{k}
u^{*}_{k})(w_k)}\\
\nonumber&=& \nabla^2\skalp{ \lambda_k,g }(x_{k})w_k + \nabla
g(x_{k})^{T}\hat{D}^{*}\hat{N}_{D}(g(\bar{x})+
\vartheta_{k}h_{k}, \bar{\lambda}+\vartheta_{k}\xi_{k})(\nabla g(x_{k})w_k).
\end{eqnarray}
 We may now use the argumentation from the proof of Theorem \ref{ThGraphDer} to show that $\xi_{k}$
converges to the unique $\bar{\xi}$ satisfying (\ref{eq-3.14}). Taking now the outer set limits for $k
\rightarrow \infty$ on both sides of (\ref{eq-3.17}), we obtain that $w^\ast \in D^{*}
\hat{N}_{\Gamma}((\bar{x},\bar{x}^{*}); (u,u^{*}))(w)$
if and only if it admits the representation
 $$w^\ast \in \nabla^{2}\langle \bar{\lambda},g\rangle
(\bar{x})w+ \nabla g(\bar{x})^{T}D^{*}\hat{N}_{D} ((g(\bar{x}),\bar{\lambda});(\nabla g(\bar{x})u,
\bar{\xi}))(\nabla g(\bar{x})w)$$
with $\bar{\lambda}$ and $\bar{\xi}$ specified above.
\qed\end{proof}
\begin{remark}
  Setting $(u,u^\ast)=(0,0)$, we recover in this way the formula
  \begin{eqnarray*}
D^{*}\hat{N}_{\Gamma}(\xb,\xba)(w)
\nonumber& = & \nabla^{2}\skalp{\lb, g} (x)w +
  \nabla g(\bar{x})^{T}D^{*}\hat{N}_{D}(g(\bar{x}),\bar{\lambda})(\nabla g (\bar{x})w),
\end{eqnarray*}
which has been derived in \cite{OR11} under the standard reducibility and nondegeneracy assumptions
from \cite{BS}. This formula thus holds also under the weakened assumptions (A1), (A2).
\end{remark}
Under the additional assumptions, mentioned in Section 3.1, relations (\ref{eq-3.14}) can be
simplified. In particular, under the PDC condition at $g(\bar{x})$, the first relation from
(\ref{eq-3.14}) reduces to (\ref{eq-3.8}) (with $\xi$ replaced by $\bar{\xi}$).


\section{Main results}
On the basis of Theorems \ref{ThImplMapping}, \ref{ThGraphDer} and \ref{ThDirLimCoderiv} we may now
state our main result - a new criterion for the Aubin of the solution map of a variational system,
given by (\ref{eq-99}), (\ref{eq-104})  around a specified  reference point. \\

\begin{theorem}\label{ThMain}

Let $0\in M(\bar{p},\bar{x})$ with $M$ specified by (\ref{eq-104}),   the assumptions (A1), (A2) be
fulfilled and let $\bar{\lambda}$ be the (unique)  multiplier satisfying (\ref{eq-3.55}) with $\xba=
-H(\bar{p},\bar{x})$. Further assume that 
\begin{enumerate}
 \item [(i)]
 for any $q \in \mathbb{R}^{l}$ the variational system
\begin{equation}\label{eq-4.1}
\begin{split}
0=\nabla_{p}H(\bar{p},\bar{x})q + \nabla_{x}\mathcal{L}(\bar{p},\bar{x},\bar{\lambda})u + \nabla g
(\bar{x})^{T}\xi\\ (\nabla g(\bar{x})u, \xi)\in T_{\Gr \hat{N}_{D}}(g(\bar{x}),\bar{\lambda})
\end{split}
\end{equation}
has a solution $(u,\xi)\in \mathbb{R}^{n}\times \mathbb{R}^{l}$;
\item [(ii)]
$M$ is metrically subregular at $(\bar{p},\bar{x})$, and
\item [(iii)]
for any nonzero $(q,u)$ satisfying (with a corresponding unique  $\xi$) relations (\ref{eq-4.1})
one has the implication 
\begin{equation}\label{eq-4.2}
\begin{split}
0 \in \nabla_{x}\mathcal{L}(\bar{p},\bar{x}, & \bar{\lambda})^{T} v^{*}+ \nabla
g(\bar{x})^{T}D^{*}\hat{N}_{D}((g(\bar{x}), \bar{\lambda}); (\nabla g(\bar{x}), \xi)) (\nabla
g(\bar{x})v^{*})\\ & \Rightarrow v^{*}\in \ker \nabla_{p}H(\bar{p},\bar{x})^{T}.
\end{split}
\end{equation}
\end{enumerate}
Then the respective $S$ has the Aubin property around $(\bar{p},\bar{x})$ and for any
$q \in \mathbb{R}^{l}$
\begin{equation}\label{eq-4.3}
\begin{split}
DS(\bar{p},\bar{x})(q)=\{u | \exists \xi  & : (\nabla g (\bar{x})u,\xi) \in T_{\Gr \hat{N}_{D}}
(g(\bar{x}),\bar{\lambda}),\\ 0 & = \nabla_{p}H(\bar{p},\bar{x})q + \nabla_{x}
\mathcal{L}(\bar{p},\bar{x},\bar{\lambda})^{T}u + \nabla g(\bar{x})^{T}\xi\}.
\end{split}
\end{equation}
The above assertions remain true provided assumptions (ii), (iii) are replaced by
\begin{enumerate}
\item [(iv)]
for any nonzero $(q,u)$ satisfying (with a corresponding unique $\xi$) relations (\ref{eq-4.1}) one
has the implication
\begin{equation}\label{eq-4.4}
\begin{split}
0 \in \nabla_{x}\mathcal{L}(\bar{p},\bar{x},  \bar{\lambda})^{T} v^{*}+ \nabla
g(\bar{x})^{T}D^{*}\hat{N}_{D}((g(\bar{x}), \bar{\lambda})& ;(\nabla g(\bar{x})u, \xi))
(\nabla g(\bar{x})v^{*})\\ & \Rightarrow v^{*} = 0.
\end{split}
\end{equation}
\end{enumerate}
\end{theorem}
 The proof follows easily from Theorems \ref{ThImplMapping}, \ref{ThGraphDer} and
\ref{ThDirLimCoderiv} and relations (\ref{eq-106}), (\ref{eq-107}). By imposing the additional
assumptions, mentioned in Section 3.1, formulas (\ref{eq-4.1}) and (\ref{eq-4.3}) can be appropriately
simplified. In particular, when $D$ is convex polyhedral, then (\ref{eq-4.1}) attains the form of the
generalized equation (GE) 
\begin{equation}\label{eq-4.5}
0 = \nabla_{p}H(\bar{p},\bar{x})q + \nabla_{x}\mathcal{L}(\bar{p},\bar{x},\bar{\lambda})u + \nabla g
(\bar{x})^{T}\xi, ~ \xi \in N_{K}(\nabla g(\bar{x})u).
\end{equation}
Denoting now $w :=(q,u)$ and $\Lambda := \mathbb{R}^{l} \times (\nabla g(\bar{x}))^{-1}K$,
(\ref{eq-4.5}) amounts to  the homogenous affine variational inequality
\begin{equation}\label{eq-110}
0 \in \left[ \begin{array}{cc} 0 & 0\\
\nabla_{p}H(\bar{p},\bar{x}), & \nabla_{x}\mathcal{L}(\bar{p},\bar{x},\bar{\lambda})
\end{array} \right] w + N_{\Lambda}(w).
\end{equation}
Indeed, thanks to the polyhedrality of $D$, $K$ is also polyhedral and
\[
N_{\Lambda}(w)=N_{\mathbb{R}^{l}}(q) \times \nabla g(\bar{x})^{T} N_{K}(\nabla g(\bar{x})u)
\]
 without any qualification conditions. For the solution of (\ref{eq-110}) various methods are available,
cf. \cite{FP}. This case will now be illustrated by an academic example.
\begin{example}\label{ExOne}
Consider the solution map $S:\mathbb{R} \rightrightarrows  \mathbb{R}^{2}$ of the GE 
\begin{equation}\label{eq-4.6}
0 \in M(p,x)=\left[
\begin{array}{l}
x_{1} - p\\ -x_{2}+x^{2}_{2}
\end{array}\right] + \hat{N}_{\Gamma}(x)
\end{equation}
with $\Gamma$ given by $D = \mathbb{R}^{2}_{-}$ and
\[
g(x)=\left[
\begin{array}{l}
g_{1} (x)\\ g_{2}(x)
\end{array}\right] =
\left[\begin{array}{l}
0.5 x_{1}-0.5 x^{2}_{1} - x_{2}\\ 0.5 x_{1} - 0.5x^{2}_{1} + x_{2}
\end{array}\right].
\]
Clearly, $\Gamma$ is a nonconvex set depicted in Fig,1. Let $(\bar{p},\bar{x})=(0,(0,0))$ be the
reference point. Since $\Gamma$ fulfills LICQ at $\bar{x}$, we conclude that assumptions (A1), (A2) are
fulfilled. Clearly, $x^{*}= -H(\bar{p}, \bar{x}) = (0,0)$ and $\bar{\lambda} = (0,0)$ as well. By
virtue of the polyhedrality of $D$ the variational system (\ref{eq-4.1}) attains the form
(\ref{eq-4.5}). In our case it amounts to
\begin{equation}\label{eq-4.7}
0 = \left[
\begin{array}{r}
-q\\ 0
\end{array}\right]  +
\left[
\begin{array}{r}
u_{1}\\ -u_{2}
\end{array}\right]  +
\left[
\begin{array}{rc}
0.5 & 0.5\\ -1 & 1
\end{array}\right]\xi, ~
\xi \in N_{\mathbb{R}^{2}_{-}}
\left( \left[
\begin{array}{l}
0.5 u_{1} - u_{2}\\ 0.5 u_{1} + u_{2}
\end{array}\right]\right),
\end{equation}
 because $K= T_{D}(g(\bar{x}))\cap [\bar{\lambda}]^{\perp} =D$.

\begin{figure}[!ht]\label{figure1}
\begin{center}
\input{gamma.tex}
\end{center}
\caption{Set $\Gamma$.}
\end{figure}

It is not difficult to compute that for $q \leq 0$ one has three different solutions $(u,\xi)$ of
(\ref{eq-4.7}), namely 
\begin{align}
& u_{1}= q,~ u_{2}=0,~ \xi_{1}=0,~ \xi_{2}=0 \label{eq-4.8}\\ & u_{1}= \frac{4}{3}q,~ u_{2}=
-\frac{2}{3}q,~ \xi_{1}=0,~ \xi_{2}=-\frac{2}{3}q \label{eq-4.9}\\ & u_{1}= \frac{4}{3}q,~
u_{2}=\frac{2}{3}q,~ \xi_{1}= -\frac{2}{3}q,~ \xi_{2}= 0, \label{eq-4.10}
\end{align}
and for $q \geq 0$ we have the unique solution
\begin{align}
 u_{1}= u_{2} =0,~ \xi_{1}= \xi_{2}= q.  \label{eq-4.11}
\end{align}
So, assumption (i) of Theorem \ref{ThMain} is fulfilled and we know the critical directions
$(q,u)\neq 0$ for which the implication (\ref{eq-4.4}) will be examined. Starting with (\ref{eq-4.8}),
one has $\nabla g(\bar{x})u=(0.5q, 0.5q)$ and
\[
\begin{array}{l}
D^{*}N_{\mathbb{R}^{2}_{-}}
\left(\left[ \begin{array}{l}
0\\ 0
\end{array}\right],
\left[ \begin{array}{l}
0\\ 0
\end{array}\right]\right);
\left(\left[ \begin{array}{l}
0.5q\\ 0.5q
\end{array}\right],
\left[ \begin{array}{l}
0\\ 0
\end{array}\right]\right)
\left(\left[ \begin{array}{l}
0.5 v^{*}_{1}-v^{*}_{2}\\ 0.5 v^{*}_{1}+v^{*}_{2}
\end{array}\right]\right)=\\
\\
\left[ \begin{array}{l}
D^{*}N_{\mathbb{R}_{-}} ((0,0); (0.5q, 0))(0.5 v^{*}_{1}-v^{*}_{2})\\ D^{*}N_{\mathbb{R}_{-}} ((0,0);
(0.5q, 0))(0.5 v^{*}_{1}+v^{*}_{2})
\end{array}\right]=
\left[ \begin{array}{l}
0\\ 0
\end{array}\right]
\end{array}
\]
 by virtue of the definition and \cite[Proposition 6.41]{RoWe98}. The left-hand side of
(\ref{eq-4.4}) reduces to the linear system in variables $(v^{*},\eta)\in \mathbb{R}^{2}\times
\mathbb{R}^{2}$
\[
0=
\left[
\begin{array}{r}
v^{*}_{1}\\ -v^{*}_{2}
\end{array}\right]  +
\left[
\begin{array}{rr}
0.5 & 0.5\\ -1 & 1
\end{array}\right]\eta, ~ \eta = 0,
\]
verifying the validity of implication (\ref{eq-4.4}). In the case (\ref{eq-4.9}),
 $\nabla
g(\bar{x})u=(\frac{4}{3}q,0)$ and
\[
D^{*}N_{\mathbb{R}^{2}_{-}}
\left(\left(\left[ \begin{array}{l}
0\\ 0
\end{array}\right],
\left[ \begin{array}{l}
0\\ 0
\end{array}\right]\right);
\left(\left[ \begin{array}{c}
\frac{4}{3}q\\
0
\end{array}\right],
\left[ \begin{array}{c}
0\\ -\frac{2}{3}q
\end{array}\right]\right)\right)
\left(\left[ \begin{array}{l}
0.5 v^{*}_{1}-v^{*}_{2}\\ 0.5 v^{*}_{1}+v^{*}_{2}
\end{array}\right]\right)= \{0\} \times \mathbb{R}
\]
 provided $v^{*}_{2}=-0.5 v^{*}_{1}$. The respective linear system in variables $(v^{*},\eta)$
reduces to
\[
0=
\left[
\begin{array}{r}
v^{*}_{1}
\\ 0.5 v^{*}_{1}
\end{array}\right]  +
\left[
\begin{array}{rr}
0.5 & 0.5\\ -1 & 1
\end{array}\right]
\left[
\begin{array}{r}
0\\
\eta
\end{array}\right],
\]
verifying again the validity of (\ref{eq-4.4}). In the same way we compute that in the case
(\ref{eq-4.10}) one has  $\nabla g(\bar{x})u=(0,\frac{4}{3}q)^{T}$ and
\[
D^{*}N_{\mathbb{R}^{2}_{-}}
\left(\left(\left[ \begin{array}{l}
0\\ 0
\end{array}\right],
\left[ \begin{array}{l}
0\\ 0
\end{array}\right]\right);
\left(\left[ \begin{array}{c}
0\\
\frac{4}{3}q
\end{array}\right],
\left[ \begin{array}{c}
-\frac{2}{3}q\\ 0
\end{array}\right]\right)\right)
\left(\left[ \begin{array}{l}
0.5 v^{*}_{1}-v^{*}_{2}\\ 0.5 v^{*}_{1}+v^{*}_{2}
\end{array}\right]\right)= \mathbb{R} \times \{0\}
\]
provided $v^{*}_{2}=0.5 v^{*}_{1}$. Taking this into account, we arrive at the linear system
\[
0=
\left[
\begin{array}{r}
v^{*}_{1}\\
-0.5 v^{*}_{1}
\end{array}\right]  +
\left[
\begin{array}{rr}
0.5 & 0.5\\ -1 & 1
\end{array}\right]
\left[
\begin{array}{r}
\eta\\
0
\end{array}\right],
\]
showing that $v^{*}=0$. Finally, concerning the last case (\ref{eq-4.11}),
 $\nabla g(\bar{x})u=(0,0)$ and
\[
D^{*}N_{\mathbb{R}^{2}_{-}}
\left(\left(\left[ \begin{array}{l}
0\\ 0
\end{array}\right],
\left[ \begin{array}{l}
0\\ 0
\end{array}\right]\right);
\left(\left[ \begin{array}{c}
0\\ 0
\end{array}\right],
\left[ \begin{array}{c}
q\\ q
\end{array}\right]\right)\right)
\left(\left[ \begin{array}{l}
0.5 v^{*}_{1}-v^{*}_{2}\\ 0.5 v^{*}_{1}+v^{*}_{2}
\end{array}\right]\right)= \mathbb{R} \times \mathbb{R},
\]
 provided $v^{*}_{1}=0.5 v^{*}_{2}$ and, at the same time, $v^{*}_{1} =-0.5 v^{*}_{2}$. This
imediately implies that $v^{*}=0$ and we are done. On the basis of Theorem \ref{ThMain} we have shown
that the implicit multifunction $S$ generated by (\ref{eq-4.6}) has the Aubin property around $(0,0)$
and, for a given $q$, $DS(0,0)(q)$ is the set of solutions to (\ref{eq-4.7}).

Next we show that this result cannot be obtained via the Mordukhovich criterion and the standard
calculus, which amounts to proving that the ``standard'' adjoint GE (cf.\cite[Corollary 4.61]{M1})
possesses only the trivial solution. Indeed, this GE amounts in our case to

\begin{equation}\label{eq-4.12}
0 \in
\left[
\begin{array}{r}
v^{*}_{1}\\ -v^{*}_{2}
\end{array}\right]  +
\left[
\begin{array}{rr}
0.5 & 0.5\\ -1 & 1
\end{array}\right]
 D^{*}N_{\mathbb{R}^{2}_{-}}
\left(\left[ \begin{array}{l}
0\\ 0
\end{array}\right],
\left[ \begin{array}{l}
0\\ 0
\end{array}\right]\right)
\left(\left[ \begin{array}{l}
0.5 v^{*}_{1}-v^{*}_{2}\\ 0.5 v^{*}_{1}+v^{*}_{2}
\end{array}\right]\right)
\end{equation}
and it is easy to check that, e.g., $v^{*}=(-0.5, 1)^{T}$ is a solution of (\ref{eq-4.12}).
Consequently, the Aubin property of $S$ cannot be detected  in this way.\hfill $\triangle$
\end{example}
\section{Variational systems with conic constraint sets}

In this concluding section we will consider a variant of Theorem \ref{ThMain} under the additional
assumption that $D$ is a closed convex cone with vertex at $0$ and $P_{D}(\cdot)$ is directionally
differentiable over $\mathbb{R}^{s}$. As implied by (\ref{eq-3.7}), the variational system
(\ref{eq-4.1}) attains then the form 
\begin{equation}\label{eq-5.1}
\begin{split}
0 = & \nabla_{p}H(\bar{p},\bar{x})q + \nabla_{x}\mathcal{L}(\bar{p},\bar{x},\bar{\lambda})u+\nabla
g(\bar{x})^{T}\xi\\ & \nabla g(\bar{x})u = P^{\prime}_{D}(g(\bar{x})+\bar{\lambda}; \nabla
g(\bar{x})u+\xi)
\end{split}
\end{equation}
which, under the PDC condition at $g(\bar{x})$, further simplifies to the form (\ref{eq-4.5}). If
$D$ is the Carthesian product of Lorentz cones or the L\"{o}wner cone (\cite{BS}), then we  dispose
with an efficient formula for $P^{\prime}_{D}(\cdot ; \cdot)$ which depends on the position of
$(g(\bar{x}), \bar{\lambda})$ in $\Gr N_{D}$, cf. \cite[Lemma 2]{OS} and \cite[Theorem 4.7]{SS}.

Concerning the GE on the left-hand side of (\ref{eq-4.2}) or (\ref{eq-4.4}), it is advantageous to
rewrite it in terms of $P_{D}$ (instead of $N_{D}$). Let $(\bar{a},\bar{b}) \in \Gr N_{D}$. Since
\[
\Gr N_{D} = \left\{(a,b)\in \mathbb{R}^{s} \times \mathbb{R}^{s} \left| \left(
\begin{array}{c}
a+b\\ a
\end{array}
\right) \in \Gr P_{D}\right.\right\},
\]
one has, by virtue of \cite[Theorem 1.17]{M1}, that
\[
p \in \hat{D}^{*}N_{D}(a,b)(q) \Longleftrightarrow -q \in \hat{D}^{*}P_{D}(a+b,a)(-q-p)
\]
for any $(p,q)\in \mathbb{R}^{s} \times \mathbb{R}^{s}$. It follows that the GE on the left-hand
side of (\ref{eq-4.2}) can be equivalently written down as the system
\begin{align}
& 0 = \nabla_{x}\mathcal{L}(\bar{p},\bar{x},\bar{\lambda})^{T}v^{*}+ \nabla g(\bar{x})^{T}(d-\nabla
g(\bar{x})v^{*})  \label{eq-5.2}\\ & -\nabla g(\bar{x})v^{*} \in D^{*}P_{D} ((g(\bar{x})+\bar{\lambda},
g(\bar{x})); (\nabla g(\bar{x})u+\xi, \nabla g(\bar{x})u))(-d)\label{eq-5.3}
\end{align}
in variables $(v^{*},d) \in \mathbb{R}^{n} \times \mathbb{R}^{s}$. If $D$ is the Carthesian product
of Lorentz cones or the  L\"{o}wner cone, then the directional limiting coderivative of $P_{D}$ can be
computed by using Definition \ref{DefVarGeom}(ii) and the formulas for regular coderivatives of $P_{D}$
in \cite{OS}  and \cite{Ding}, respectively. For illustration consider the case when $D$ amounts to
just one Lorentz cone in $\mathbb{R}^{s}$, i.e.,
\[
D=\mathcal{K}:=\{(z_{0},\bar{z}) \in \mathbb{R} \times \mathbb{R}^{s-1}| z_{0}\geq \| \bar{z} \|\}.
\]
 We will analyze here only the most difficult situation when  $g(\bar{x})=0$ and $\bar{\lambda}=0$
and provide formulas for the directional limiting coderivatives of $P_{\mathcal{K}}$ at $(0,0)$ for all
possible nonzero directions from
\begin{equation}\label{eq-5.4}
T_{\Gr P_{\mathcal{K}}}(0,0) = \{(h,k)| k\in P_{\mathcal{K}}(h)\},
\end{equation}
see \cite[Lemma 2(iv)]{OS}. We have thus to distinguish between the following five situations:
\begin{align}
\bullet &~~ ~~h \in \inn \mathcal{K}, ~ k=h;\label{eq-5.5}\\
\bullet & ~~~~ h \in \inn \mathcal{K}^{\circ},~ k=0;\label{eq-5.6}\\
\bullet & ~~~~ h \not\in  \mathcal{K} \cup \mathcal{K}^{\circ}, k= P_{\mathcal{K}}(h);\label{eq-5.7}\\
\bullet &~~ ~~h \in \bd \mathcal{K}, ~ k=h;\label{eq-5.8}\\
\bullet & ~~~~ h \in \bd \mathcal{K}^{\circ},~ k=0.\label{eq-5.9}
\end{align}
 In the cases (\ref{eq-5.5}), (\ref{eq-5.6})  we get immediately from  \cite[Lemma 1(iv)]{OS} the
formulas
\begin{eqnarray}
 &  D^{*}P_{\mathcal{K}} ((0,0); (h,k))(u^{*}) = u^{*},\label{eq-13}\\
 & D^{*}P_{\mathcal{K}} ((0,0); (h,k))(u^{*}) = 0, \label{eq-14}
\end{eqnarray}
 respectively. Likewise, in the case (\ref{eq-5.7}) one has
\begin{eqnarray}
 &  D^{*}P_{\mathcal{K}} ((0,0); (h,k))(u^{*}) = \{C(w,\alpha) u^{*}| w \in \mathbb{S}_{n-1}, \alpha
 \in [0,1]\}, \label{eq-15}
\end{eqnarray}
where
\[
C(w,\alpha)=\frac{1}{2}\left [
\begin{array}{cc}
2 \alpha I+(1-2\alpha)w w^{T} & w\\ w^{T} & 1
\end{array}\right ].
\]
Concerning the case (\ref{eq-5.8}), by passing to subsequences if necessary,  one may have
sequences $(h_{i}, k_{i})\longsetto{{\Gr P_{\mathcal{K}}}}
 (h,k), \lambda_{i}\searrow 0$ such that for
$i$ sufficiently large one of the following three situations occurs:
\begin{itemize}
 \item[$\ast$]
 $h_{i} \not\in ~\mathcal{K} \cup \mathcal{K}^{0}~ (k_{i} = P_{\mathcal{K}}(h_{i}))$;
 \item[$\ast$]
 $h_{i} \in {\rm int}~\mathcal{K} ~(k_{i} = h_{i})$;
 \item[$\ast$]
  $h_{i} \in bd~\mathcal{K} ~(k_{i} = h_{i})$.
\end{itemize}
Correspondingly, we obtain from \cite[Lemma 1(iv) and Theorem 4]{OS},  that
\begin{equation}\label{eq-16}
D^{*}P_{\mathcal{K}}  ((0,0); (h,k))(u^{*}) = \{C(w,\alpha) u^{*}| w \in \mathbb{S}_{n-1}, \alpha \in
[0,1]\} ~ \cup
 \bigcup\limits_{A \in \mathcal{A}(u^{*})} {\rm conv} \{u^{*}, A u^{*}\},
\end{equation}
 where
\[
\mathcal{A}(u^{*}):= \left \{ \left. I + \frac{1}{2}\left [
\begin{array}{cr}
-w w^{T} & w\\ w^{T} & -1
\end{array}\right ] \right | w \in \mathbb{S}_{n-1},
\left \langle
\left [
\begin{array}{c}
-w\\ 1
\end{array}
\right ], u^{*}
\right \rangle \geq 0
\right \}.
\]
Analogously, in the  case (\ref{eq-5.9}),  by passing to subsequences if necessary,  one may have
sequences $(h_{i}, k_{i})\longsetto{{\Gr P_{\mathcal{K}}}}(h,k), \lambda_{i}\searrow 0$ such that for
$i$ sufficiently large one of the following three situations occurs:
\begin{itemize}
 \item[$\ast$]
 $h_{i} \not\in ~\mathcal{K} \cup ~\mathcal{K}^{0} ~(k_{i} = P_{\mathcal{K}}(h_{i}))$;
 \item[$\ast$]
 $h_{i} \in {\rm int}~\mathcal{K}^{0} ~(k_{i} = 0)$;
 \item[$\ast$]
  $h_{i} \in bd~\mathcal{K}_{0} ~(k_{i} = 0)$.
\end{itemize}
 Correspondingly, we obtain from \cite[Lemma 1(iv) and Theorem 4]{OS} that
\begin{equation}\label{eq-17}
D^{*}P_{\mathcal{K}} ((0,0); (h,k))(u^{*})=\{C(w,\alpha) u^{*}| w \in \mathbb{S}_{n-1}, \alpha \in
[0,1]\} ~\cup
 \bigcup\limits_{B \in \mathcal{B}(u^{*})} {\rm conv} \{u^{*}, B u^{*}\},
\end{equation}
 where
\[
\mathcal{B}(u^{*}):= \left \{ \left. \frac{1}{2}\left [
\begin{array}{cr}
w w^{T} & w\\ w^{T} & 1
\end{array}\right ] \right | w \in \mathbb{S}_{n-1},
\left \langle
\left [
\begin{array}{c}
w\\ 1
\end{array}
\right ], u^{*}
\right \rangle \geq 0
\right \}.
\]
Next we illustrate the above described procedure via a conic reformulation of \cite[Example 5]{GO3}.
\begin{example}\label{ExTwo} Consider the solution map $S:\mathbb{R} \rightrightarrows\mathbb{R}^{2}$
of the GE given by (\ref{eq-99}), (\ref{eq-104})  with
\[
H(p,x)=\left[
\begin{array}{c}
x_{1} - p\\ -x_{2}
\end{array}\right], \quad
 g(x)=\left[
\begin{array}{l}
2x_{2} \\ -x_{1}
\end{array}\right]
\]
 and $D=\mathcal{K}$ being the Lorentz cone in $\mathbb{R}^{2}$. Let $(\bar{p},\bar{x})= (0,(0,0))$
be the reference point so that $\bar{\lambda}=(0,0)$. It is easy to see that assumptions (A1), (A2) are
fulfilled and, since the Lorentz cone in $\mathbb{R}^{2}$ is a polyhedral set, instead of
(\ref{eq-5.1}) we can compute the ``critical'' directions via (\ref{eq-4.5}). The variational system
(\ref{eq-4.5}) attains the form
\begin{equation}\label{eq-5.10}
0 = \left[
\begin{array}{r}
-q\\ 0
\end{array}\right]  +
\left[
\begin{array}{r}
u_{1}\\ -u_{2}
\end{array}\right]  +
\left[
\begin{array}{rr}
0 & -1\\ 2 & 0
\end{array}\right]\xi, ~
\xi \in N_{\mathcal{K}}
\left( \left[
\begin{array}{l}
2 u_{2}\\ - u_{1}
\end{array}\right]\right).
\end{equation}
It is not difficult to compute that for $q \leq 0$ one has three different solutions $(u,\xi)$ of
(\ref{eq-5.10}), namely 
\begin{align}
& u_{1}= q,~ u_{2}=0,~ \xi_{1}=0,~ \xi_{2}=0 \label{eq-5.11}\\ & u_{1}= \frac{4}{3}q,~ u_{2}=
-\frac{2}{3}q,~ \xi_{1}=-\frac{1}{3}q,~ \xi_{2}=\frac{1}{3}q \label{eq-5.12}\\ & u_{1}= \frac{4}{3}q,~
u_{2}=\frac{2}{3}q,~ \xi_{1}= \frac{1}{3}q,~ \xi_{2}= \frac{1}{3}q \label{eq-5.13}
\end{align}
and for $q \geq 0$ one has the unique solution
\begin{equation}\label{eq-5.14}
u_{1}=u_{2}=0, ~~\xi_{1}=0, ~~ \xi_{2}=-q.
\end{equation}
So, assumption (i) of Theorem \ref{ThMain} is fulfilled and we will check assumption (iv). Starting
with (\ref{eq-5.11}), system (\ref{eq-5.2}), (\ref{eq-5.3}) attains the form
\begin{align}
 0 =
\left[
\begin{array}{r}
v^{*}_{1}\\ -v^{*}_{2}
\end{array}\right]  +
\left[
\begin{array}{cc}
0 & -1\\ 2 & 0
\end{array}\right]
\left[ \begin{array}{ccc}
1 & & 0
\\ 0 & & 4
\end{array}\right]v^{*}=
\left[ \begin{array}{l}
-d_2\\ -5 v^{*}_{2} + 2d_{1}
\end{array}\right] \label{eq-5.15}\\[1ex]
\left[
\begin{array}{r}
-2 v^{*}_{2}\\ v^{*}_{1}
\end{array}\right] \in D^{*}P_{\mathcal{K}}
\left((0,0);
\left(\left[ \begin{array}{r}
0\\ -q
\end{array}\right],
\left[ \begin{array}{r}
0\\ -q
\end{array}\right]\right)\right)(-d).\label{eq-5.16}
\end{align}
 By virtue of formula (\ref{eq-13}) this system reduces to the equations
 $$d_{2}=0, ~ d_{1}=
\frac{5}{2}v^{*}_{2}, ~ v^{*}_{1}=0, ~ 2 v^{*}_{2}= d_{1},$$
verifying that $v^{*}=0$. In the case
(\ref{eq-5.12}), one arrives at the equation (\ref{eq-5.15}) together with the relation
\begin{align}\label{eq-5.17}
\left[
\begin{array}{r}
-2 v^{*}_{2}\\ v^{*}_{1}
\end{array}\right] \in D^{*}P_{\mathcal{K}}
\left((0,0);
\left(\left[ \begin{array}{r}
-\frac{5}{3}q\\ -q
\end{array}\right],
\left[ \begin{array}{r}
-\frac{4}{3}q\\ -\frac{4}{3}q
\end{array}\right]\right)\right)(-d).
\end{align}
 Now we have to employ formula (\ref{eq-15}). For $w=-1$ one obtains from (\ref{eq-5.17}) the
equation
\[
\left[
\begin{array}{r}
-2 v^{*}_{2}\\ v^{*}_{1}
\end{array}\right] = -
\left[
\begin{array}{rr}
0.5 & -0.5\\ -0.5 & 0.5
\end{array}\right]d
\]
which, together with (\ref{eq-5.15}),  implies that $v^{*}=0$. For $w=1$ one obtains from
(\ref{eq-5.17}) the  equation
\[
\left[
\begin{array}{r}
-2 v^{*}_{2}\\ v^{*}_{1}
\end{array}\right] = -
\left[
\begin{array}{cc}
0.5 & 0.5\\ 0.5 & 0.5
\end{array}\right]d
\]
that again implies that $v^\ast=0$. Thus
 the case (\ref{eq-5.12}) is completed. Likewise, in the remaining cases (\ref{eq-5.13}),
 (\ref{eq-5.14}) we apply the formulas (\ref{eq-15}) and (\ref{eq-14}), respectively, and verify again
 that in all solutions of the respective system (\ref{eq-5.2}), (\ref{eq-5.3}) one has $v^{*}=0$. The examined solution map
 $S$ has thus the Aubin property around $(\bar{p},\bar{x})$. Note that, as in Example \ref{ExOne}, this
 conclusion cannot be made on the basis of the standard criteria.

\hfill$\triangle$
\end{example}

\section{Concluding remarks}

The formulas provided in the second part of Section 5 for $D$ being the Lorentz cone could easily be
extended to the case when $D$ amounts to the Carthesian product of several Lorentz cones. Further, on
the basis of \cite{Ding} one could compute the directional limiting coderivatives of the projection
mapping onto the L\"{o}wner cone which would enable us to apply the presented theory also to
parameterized semidefinite programs. Finally, one could think of variational systems, not having the
(relatively simple) structure (\ref{eq-104}). For example,  $p$ could arise also in the constraints or
one could consider implicit constraints like in quasi-variational inequalities \cite{QVI}. All these
situations offer an interesting topic for a future research.

\section*{Acknowledgements}
 The research of the first author was supported by the Austrian Science Fund (FWF) under grant
 P29190-N32. The research of the second author was supported by the Grant Agency of the Czech Republic,
 project 15-00735S and the Australian Research Council, project  DP160100854.

\end{document}

%% file: gamma.tex
\begin{picture}(0,0)%
\includegraphics{gamma.jpg}%
\end{picture}%
\setlength{\unitlength}{4144sp}%
\begingroup\makeatletter\ifx\SetFigFont\undefined
\def\x#1#2#3#4#5#6#7\relax{\def\x{#1#2#3#4#5#6}}%
\expandafter\x\fmtname xxxxxx\relax \def\y{splain}%
\ifx\x\y   
\gdef\SetFigFont#1#2#3{%
  \ifnum #1<17\tiny\else \ifnum #1<20\small\else
  \ifnum #1<24\normalsize\else \ifnum #1<29\large\else
  \ifnum #1<34\Large\else \ifnum #1<41\LARGE\else
     \huge\fi\fi\fi\fi\fi\fi
  \csname #3\endcsname}%
\else
\gdef\SetFigFont#1#2#3{\begingroup
  \count@#1\relax \ifnum 25<\count@\count@25\fi
  \def\x{\endgroup\@setsize\SetFigFont{#2pt}}%
  \expandafter\x
    \csname \romannumeral\the\count@ pt\expandafter\endcsname
    \csname @\romannumeral\the\count@ pt\endcsname
  \csname #3\endcsname}%
\fi
\fi\endgroup
\begin{picture}(3174,3666)(3814,-4144)
\put(4608,-1405){\makebox(0,0)[lb]{\smash{\SetFigFont{12}{14.4}{rm}$\Gamma$}}}
\end{picture}